\documentclass[10pt, a4paper, english,reqno]{amsart}
\usepackage{amsmath} 
\usepackage{amsthm} 
\usepackage{amssymb} 
\usepackage[dvipsnames]{xcolor}
\usepackage{amscd} 
\usepackage{mathtools}

\usepackage{subcaption}

\usepackage{mlmodern}
\usepackage{array} 
\usepackage[cal=boondoxo,scr=euler]{mathalfa}
\usepackage[backref=page,linktocpage]{hyperref} 
\usepackage{cleveref} 
\usepackage{caption} 
\usepackage{graphics,graphicx} 
\usepackage{tikz,tikz-cd} 

\usetikzlibrary{graphs,graphs.standard,calc}

\usetikzlibrary{decorations.pathreplacing,}

\usepackage{enumerate} 

\DeclareMathAlphabet{\mathsf}{OT1}{\sfdefault}{m}{n}

\newcommand{\nocontentsline}[3]{}
\newcommand{\tocless}[2]{\bgroup\let\addcontentsline=\nocontentsline#1{#2}\egroup}

\usepackage[margin=1.5in]{geometry}

\usepackage{verbatim}

\usepackage{scalerel}

\makeatletter
\def\dual#1{\expandafter\dual@aux#1\@nil}
\def\dual@aux#1/#2\@nil{\begin{tabular}{@{}c@{}}#1\\#2\end{tabular}}
\makeatother

\makeatletter
\@namedef{subjclassname@2020}{\textup{2020} Mathematics Subject Classification}
\makeatother

\DeclareMathAlphabet{\amathbb}{U}{bbold}{m}{n}

\hypersetup{
    colorlinks = true,
    linkbordercolor = {white},
    linkcolor = {MidnightBlue},
    anchorcolor = {black},
    citecolor = {MidnightBlue},
    filecolor = {cyan},
    menucolor = {MidnightBlue},
    runcolor = {cyan},
    urlcolor = {MidnightBlue}
}

\usetikzlibrary{automata}

\newtheoremstyle{teoremas}
{10pt}
{10pt}
{\itshape}
{}
{\bfseries}
{}
{.5em}
{}

\theoremstyle{teoremas}
\newtheorem{theorem}{Theorem}[section]

\newtheorem{proposition}[theorem]{Proposition}

\newtheoremstyle{definition}
{9pt}
{8pt}
{}
{}
{\bfseries}
{}
{.5em}
{}

\theoremstyle{definition}
\newtheorem{definition}[theorem]{Definition}

\newtheorem{problem}[theorem]{Problem}

\newtheorem{example}[theorem]{Example}
\newtheorem{remark}[theorem]{Remark}

\crefname{theorem}{theorem}{theorems}
\Crefname{theorem}{Theorem}{Theorems}
\crefname{lemma}{lemma}{lemmas}
\Crefname{lemma}{Lemma}{Lemmas}
\crefname{proposition}{proposition}{propositions}
\Crefname{proposition}{Proposition}{Propositions}

\DeclareMathOperator{\rk}{rk}

\newcommand{\Q}{\mathsf{Q}}

\newcommand{\Z}{\mathbb{Z}}
\newcommand{\M}{\mathsf{M}}

\newcommand{\U}{\mathsf{U}}

\newcommand{\K}{\mathsf{K}}
\newcommand{\IH}{\operatorname{IH}}
\renewcommand{\H}{\operatorname{H}}

\newcommand{\Ind}{\operatorname{Ind}}

\newcommand{\VRep}{\operatorname{VRep}}

\AtBeginDocument{
   \def\MR#1{}
}

\title[Kazhdan--Lusztig polynomials of Dowling geometries]{Kazhdan--Lusztig polynomials of\\ Dowling geometries}

\author[L. Ferroni]{Luis Ferroni}
\address{(L. Ferroni) Universit\`a di Pisa, Pisa, Italy}
\email{luis.ferroni@unipi.it}

\author[M. Larson]{Matt Larson}
\address{(M. Larson) Princeton University \& Institute for Advanced Study, Princeton, NJ, United States}
\email{mattlarson@princeton.edu}


\allowdisplaybreaks

\begin{document}

\begin{abstract}
We give a concrete combinatorial interpretation of the coefficients of the Kazhdan--Lusztig polynomials of Dowling geometries, a family of matroids which generalizes braid matroids of types $A$ and $B$. 
Furthermore, we interpret the coefficients of the equivariant Kazhdan--Lusztig polynomials and the equivariant $Z$-polynomials of Dowling geometries associated to non-trivial groups with respect to their full automorphism groups. 
\end{abstract}

\maketitle

\section{Introduction}

Inspired by the classical Kazhdan--Lusztig theory of Coxeter groups in representation theory, Elias, Proudfoot, and Wakefield initiated in \cite{elias-proudfoot-wakefield} an analogous theory for matroids. To every loopless matroid $\M$, one associates a polynomial $P_{\M}(t)\in \mathbb{Z}[t]$ satisfying:
\begin{enumerate}[\normalfont(i)]
    \item If $\rk(\M)=0$, then $P_{\M}(t)=1$.
    \item If $\rk(\M)>0$, then $\deg P_{\M}(t)<\frac{1}{2}\rk(\M)$.
    \item For every matroid $\M$, the polynomial
    \[
        Z_{\M}(t)=\sum_{F\in \mathcal{L}(\M)} t^{\rk(F)}P_{\M/F}(t)
    \]
    is palindromic\footnote{A polynomial $f(t)$ is \emph{palindromic} if $f(t)=t^d f(t^{-1})$, where $d=\deg f$.} of degree $\rk(\M)$.
\end{enumerate}
Here $\mathcal{L}(\M)$ denotes the lattice of flats of $\M$. The polynomial $P_{\M}(t)$ is called the \emph{Kazhdan--Lusztig polynomial} of $\M$, and $Z_{\M}(t)$ is called its \emph{$Z$-polynomial}. The above conditions determine $P_{\M}(t)$ and $Z_{\M}(t)$ uniquely; the existence of these polynomials was established by Proudfoot, Xu, and Young in \cite{PXY}.

As in the classical Kazhdan--Lusztig theory for Weyl groups, these polynomials have a geometric interpretation for realizable matroids. If the matroid $\M$ is represented by a linear subspace $L\subseteq \mathbb{C}^n$, then the coefficients of both $P_{\M}(t)$ and $Z_{\M}(t)$ may be interpreted in terms of intersection cohomology: the $Z$-polynomial is the intersection cohomology Poincar\'e polynomial of the closure of $L$ in $(\mathbb{P}^1)^n$, while the Kazhdan--Lusztig polynomial is the Poincar\'e polynomial of the stalk of the intersection cohomology sheaf at the most singular point $(\infty,\dotsc,\infty)$ of the closure of $L$; see \cite{elias-proudfoot-wakefield,PXY}.

The \emph{braid matroid} $\mathsf{K}_{n+1}$ is the matroid associated to the hyperplane arrangement defined by the type $A_n$ root system. Its lattice of flats is naturally identified with the lattice of set partitions of $\{0,1,\dotsc,n\}$. Since the introduction of Kazhdan--Lusztig polynomials of matroids, a major problem has been to understand these polynomials for braid matroids; see, for example, \cite[Appendix A]{elias-proudfoot-wakefield}, \cite[Section 4]{gedeon-proudfoot-young}, \cite{PY}, \cite[Section 4]{PXY}, \cite{KW}, and \cite[Example 1.4]{TOSTESON}. This problem was, in a certain sense, resolved by the authors of the present paper in \cite{ferroni-larson}. There, we gave a combinatorial interpretation of the coefficients of $P_{\mathsf{K}_{n+1}}(t)$ and $Z_{\mathsf{K}_{n+1}}(t)$ in terms of certain matroids on $[n]=\{1,\dotsc,n\}$, namely \emph{quasi series-parallel matroids}; see Section~\ref{subsec:2.1} for their precise definition. This interpretation led to explicit formulas for the generating functions of these polynomials. These results were subsequently used in \cite{gao-proudfoot-yang-zhang,proudfoot-xu-young-qsp} to obtain simple formulas for the coefficients of Kazhdan--Lusztig and $Z$-polynomials of braid matroids.

\smallskip

There is also an equivariant refinement of this theory. Suppose that a finite group $\Gamma$ acts on $L$ and preserves the hyperplane arrangement induced by the coordinate hyperplanes. Then $\Gamma$ acts on the corresponding intersection cohomology, and also on the stalk at $(\infty,\dotsc,\infty)$. This upgrades the Kazhdan--Lusztig and $Z$-polynomials to polynomials whose coefficients lie in the representation ring of $\Gamma$. In \cite{Gedeon} and \cite{PXY}, this construction is extended beyond the realizable setting: for any matroid equipped with an action of $\Gamma$, one can define the \emph{equivariant Kazhdan--Lusztig polynomial} $P^{\Gamma}_{\M}(t)$ and the \emph{equivariant $Z$-polynomial} $Z^{\Gamma}_{\M}(t)$ by a recursion analogous to the one above; see Definition~\ref{def:equivariantKL}.

The symmetric group $\mathfrak{S}_{n+1}$ acts on $\mathsf{K}_{n+1}$ by permuting the set $\{0,\dotsc,n\}$. Let $\mathfrak{S}_n$ denote the subgroup that fixes $0$. The main result of \cite{ferroni-larson} gives a formula for the coefficients of $P^{\mathfrak{S}_n}_{\mathsf{K}_{n+1}}(t)$ and $Z^{\mathfrak{S}_n}_{\mathsf{K}_{n+1}}(t)$ 
as permutation representations on sets of quasi series-parallel matroids on $[n]$. Somewhat surprisingly, this description does not seem to extend directly to the full $\mathfrak{S}_{n+1}$-equivariant Kazhdan--Lusztig and $Z$-polynomials.

\subsection*{Main result}

The purpose of this paper is to prove a generalization of the main result of \cite{ferroni-larson}, which in turn explains why the subgroup $\mathfrak{S}_n$ is the natural symmetry group in this context. Dowling introduced in \cite{dowling} a construction which associates to every finite group $G$ of order $q$ and every nonnegative integer $n$ a matroid $\Q_n(G)$ of rank $n$ on
\(
    n+q\binom{n}{2}
\)
elements. These matroids are known as \emph{Dowling geometries}. They have played an important role in matroid theory, where they form one of the two families of ``universal objects'' in the sense of \cite{KahnKung}. Equivariant invariants of Dowling geometries have also been studied in several works; see \cite{Hanlon,HanlonGeneralized,BoninAut,Schwartz,Henderson,Margolis,BibbyGadish}.

We shall use the definition of $\Q_n(G)$ given in Dowling's original paper \cite{dowling}, although many other equivalent descriptions are available; see \cite{StanleyGF,Zaslavsky,BoninBogart,BennettBogartBonin}. A \emph{partial $G$-partition} of $[n]$ consists of a subset $S\subseteq [n]$, a function $f\colon S\to G$, and a partition
\(
    S=A_1\sqcup \dotsb \sqcup A_k
\)
of $S$ into unordered nonempty blocks. Two partial $G$-partitions
$(S,f,A_1,\dotsc,A_k)$ and $(T,h,B_1,\dotsc,B_\ell)$ are declared to be equivalent if $S=T$, the two set partitions agree after relabeling of their blocks, and there exist elements $g_1,\dotsc,g_k\in G$ such that
\(
    f|_{A_i}=g_i\cdot h|_{A_i}
\)
for every $i$.

The refinement order on partial $G$-partitions is defined as follows. We say that
\(
    (S,f,A_1,\dotsc,A_k)
\)
refines
\(
    (T,h,B_1,\dotsc,B_\ell)
\)
if the latter is equivalent to a partial $G$-partition obtained from the former by a sequence of operations consisting of merging blocks and deleting blocks. This gives a partial order on partial $G$-partitions. Under this order, the minimal element is represented by the partition of $[n]$ into $n$ singleton blocks, with all choices of $f$ equivalent. The maximal element is represented by the empty subset.

Dowling proved in \cite[Theorem 3]{dowling} that the lattice of partial $G$-partitions is the lattice of flats of a matroid, which we denote by $\Q_n(G)$. As noted above, this matroid has
$n+q\binom{n}{2}$ atoms, with $n$ atoms corresponding to partial partitions with $n-1$ singleton blocks, and $q\binom{n}{2}$ atoms corresponding to partial partitions with $n-2$ singleton blocks and one block of size $2$. When $G$ is trivial, $\Q_n(G)$ is the braid matroid $\K_{n+1}$, since partitions of subsets of $[n]$ may be identified with partitions of $\{0,1,\dotsc,n\}$. When $G=\mathbb{Z}/2\mathbb{Z}$, the Dowling geometry $\Q_n(G)$ is the \emph{type B braid matroid}, namely the matroid of the type B braid arrangement, whose hyperplanes in $\mathbb{C}^n$ are given by the $n^2$ equations of the form $x_i\pm x_j = 0$ for $1\leq i \leq j \leq n$.
More generally, if $G=\mathbb{Z}/m\mathbb{Z}$, then $\Q_n(G)$ is the matroid of the reflection arrangement associated to the complex reflection group $G(m,1,n)$.

There are several natural symmetries of $\Q_n(G)$. The group $\mathfrak{S}_n$ acts by permuting the underlying set $[n]$, the group $G^n$ acts on the right on the $G$-valued functions, and $\operatorname{Aut}(G)$ acts by applying automorphisms to the values of these functions. Let $\Gamma$ be the subgroup of $\operatorname{Aut}(\Q_n(G))$ generated by these actions; see Section~\ref{ssec:automorphisms} for an explicit description. When $G$ is a nontrivial group and $n \ge 3$, Bonin \cite{BoninAut} proved that
\begin{equation}\label{eq:Gamma-is-Aut}
    \Gamma=\operatorname{Aut}(\Q_n(G)),
\end{equation}
and when $G$ is trivial, $\Gamma$ is the subgroup $\mathfrak{S}_n$ of
\(
    \operatorname{Aut}(\mathsf{K}_{n+1})=\mathfrak{S}_{n+1}.
\)

In this paper, we introduce a notion of group-labeled matroid. Let $G$ be a finite group. A \emph{$G$-labeling} of a matroid $\M$ on a ground set $E$ is a function
\(
    \phi\colon E\to G.
\)
Two $G$-labelings $\phi,\psi\colon E\to G$ are said to be equivalent if, for every connected component $C$ of $\M$, there is an element $g_C\in G$ such that
\(
    \phi|_C=g_C\cdot \psi|_C.
\)

Let $\mathcal{A}_G(n,r)$ be the set of all equivalence classes of $G$-labeled quasi series-parallel matroids on $[n]$ of rank $r$, and let $\mathcal{S}_G(n,r)$ be the subset consisting of the simple quasi series-parallel matroids. These sets carry natural actions of $\Gamma$, described in Section~\ref{ssec:automorphisms}.

\begin{theorem}\label{thm:main}
The coefficient of $t^i$ in the $\Gamma$-equivariant Kazhdan--Lusztig polynomial of $\Q_n(G)$ is the permutation representation on $\mathcal{S}_G(n,n-i)$. The coefficient of $t^i$ in the $\Gamma$-equivariant $Z$-polynomial of $\Q_n(G)$ is the permutation representation on $\mathcal{A}_G(n,n-i)$.
\end{theorem}

Passing to dimensions, we obtain the following enumerative interpretation. The coefficients of the Kazhdan--Lusztig polynomial of $\Q_n(G)$ count simple quasi series-parallel matroids on $[n]$, while the coefficients of the $Z$-polynomial count quasi series-parallel matroids on $[n]$, in both cases with a weight $q^{n-c(\M)}$
assigned to a matroid $\M$. Here $q=|G|$ and $c(\M)$ denotes the number of connected components of $\M$. In particular, the ordinary polynomials $P_{\Q_n(G)}(t)$ and $Z_{\Q_n(G)}(t)$ depend only on the order of $G$. This is consistent with the fact that Kazhdan--Lusztig and $Z$-polynomials are valuative invariants of matroids \cite[Theorem 8.9]{ArdilaSanchez}, \cite[Theorem~9.3]{ferroni-schroter}, together with the fact that the class of $\Q_n(G)$ in the symmetrized valuative group depends only on $q$ \cite[Proposition~4.8]{BoninKung}.

\subsection*{Acknowledgements} We have benefited from conversations about Kazhdan--Lusztig polynomials of braid matroids with June Huh, Trevor Karn, Jacob P. Matherne, Nick Proudfoot, Lorenzo Vecchi, and Max Wakefield. The first author is a member of the INdAM research group GNSAGA. This work was conducted while the second author was at the Institute for Advanced Study, where he is supported by the Charles Simonyi Endowment and the Oswald Veblen Fund.

\section{Background}

Throughout the paper we use the standard language and basic results of matroid theory. Our notation and conventions largely follow Oxley~\cite{oxley}.

\subsection{(Quasi)~series-parallel matroids}\label{subsec:2.1}

Series-parallel matroids form a fundamental family of matroids which originate in graph theory and in the theory of electrical networks. The term ``series-parallel'' is used in several nearby settings, including graph theory and enumerative combinatorics, where it sometimes refers to closely related but slightly different classes of objects. In order to separate these conventions, we introduced in \cite{ferroni-larson} the term \emph{quasi series-parallel} for the class of matroids which we now recall.

Two basic operations underlie the construction of series-parallel matroids. Let $\M$ and $\M'$ be matroids on ground sets $E$ and $E'$, respectively. We say that $\M'$ is a \emph{parallel extension} of $\M$ if there is an element $e\in E'$ such that $e$ belongs to a circuit of cardinality $2$ in $\M'$ and $\M' \setminus e=\M$. Dually, $\M'$ is called a \emph{series extension} of $\M$ if there is an element $e\in E'$ such that $e$ belongs to a cocircuit of cardinality $2$ in $\M'$ and $\M'/e = \M$. Notice, for example, that if $\M = \U_{1,1}$ and $\M'=\U_{2,2}$, then $\M'$ is \emph{not} a series extension of $\M$.

A \emph{series-parallel matroid} is a matroid obtained from either a single loop or a single coloop by performing a finite, possibly empty, sequence of series and parallel extensions. Equivalently, every series-parallel matroid with at least two elements is obtained from $\U_{1,2}$ by such a sequence.

\begin{definition}
    A \emph{quasi series-parallel matroid} is a matroid whose connected components are all series-parallel matroids.
\end{definition}

Several equivalent descriptions of quasi series-parallel matroids were given in \cite[Proposition~2.6]{ferroni-larson}. In particular, they can be characterized by excluded minors: a matroid $\M$ is quasi series-parallel if and only if it has no minor which is isomorphic to $\mathsf{K}_4$ or $\U_{2,4}$.

\subsection{Group-labeled matroids}

Motivated by Dowling's definition of the matroids $\Q_n(G)$, we introduce the following notion of group labeled matroids.

\begin{definition}
    Let $G$ be a finite group. A \emph{$G$-labeling} of a matroid $\M$ on a ground set $E$ is a map $\phi \colon E\to G$, where two $G$-labelings $\phi,\psi \colon E\to G$ are considered to be equivalent if, for each connected component $C$ of $\M$, there is a group element $g_C\in G$ such that $\phi|_C = g_C\cdot \psi|_C$.
\end{definition}

\begin{example}
    Consider $G\cong \Z/2\Z$, viewed as the group on two elements $\{-1,1\}$ under multiplication. The graphic matroid induced by the graph depicted in Figure~\ref{fig:example} admits exactly $16$ equivalence classes of $G$-labelings. In Figure~\ref{fig:g-labelings-example} we depict the $16$ equivalence classes of $G$-labelings, choosing representatives in such a way that the number of edges labeled with a $+1$ is always greater than the number of edges labeled with a $-1$. 
    \begin{figure}[ht]
    \centering
    \begin{tikzpicture}
        [scale=0.8,auto=center,
         every node/.style={circle, fill=black, inner sep=1.7pt},
         every edge/.append style = {thick}]
        \tikzstyle{edges} = [thick];

        \node[] (bl) at (0,0) {};
        \node[] (br) at (2,0) {};
        \node[] (tr) at (2,2) {};
        \node[] (tl) at (0,2) {};

        \draw[edges] (bl) -- (br);
        \draw[edges] (br) -- (tr);
        \draw[edges] (tr) -- (tl);
        \draw[edges] (tl) -- (bl);
        \draw[edges] (tl) -- (br);
    \end{tikzpicture}
    \caption{A graphic matroid $\M$.}
    \label{fig:example}
\end{figure}
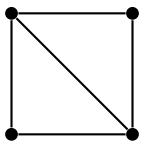
    
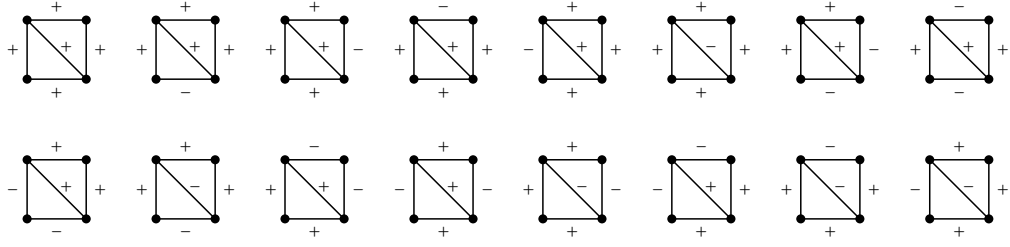
\begin{figure}[ht]
\centering

\newcommand{\signedsquarediag}[5]{%
\begin{scope}[scale=0.55]
    \tikzstyle{edges}=[thick]
    \tikzstyle{vertex}=[circle, fill=black, inner sep=1.7pt]
    \tikzstyle{sign}=[inner sep=1pt, font=\scriptsize]

    \coordinate (BL) at (0,0);
    \coordinate (BR) at (2,0);
    \coordinate (TR) at (2,2);
    \coordinate (TL) at (0,2);

    \draw[edges] (BL) -- (BR) node[midway, below=3pt, sign] {$#1$};
    \draw[edges] (BR) -- (TR) node[midway, right=3pt, sign] {$#2$};
    \draw[edges] (TR) -- (TL) node[midway, above=3pt, sign] {$#3$};
    \draw[edges] (TL) -- (BL) node[midway, left=3pt, sign] {$#4$};
    \draw[edges] (TL) -- (BR) node[pos=.55, above right=-1pt, sign] {$#5$};

    \node[vertex] at (BL) {};
    \node[vertex] at (BR) {};
    \node[vertex] at (TR) {};
    \node[vertex] at (TL) {};
\end{scope}
}

\resizebox{\textwidth}{!}{%
\begin{tikzpicture}

\begin{scope}[shift={(0,0)}]
    \signedsquarediag{+}{+}{+}{+}{+}
\end{scope}

\begin{scope}[shift={(2.4,0)}]
    \signedsquarediag{-}{+}{+}{+}{+}
\end{scope}

\begin{scope}[shift={(4.8,0)}]
    \signedsquarediag{+}{-}{+}{+}{+}
\end{scope}

\begin{scope}[shift={(7.2,0)}]
    \signedsquarediag{+}{+}{-}{+}{+}
\end{scope}

\begin{scope}[shift={(9.6,0)}]
    \signedsquarediag{+}{+}{+}{-}{+}
\end{scope}

\begin{scope}[shift={(12.0,0)}]
    \signedsquarediag{+}{+}{+}{+}{-}
\end{scope}

\begin{scope}[shift={(14.4,0)}]
    \signedsquarediag{-}{-}{+}{+}{+}
\end{scope}

\begin{scope}[shift={(16.8,0)}]
    \signedsquarediag{-}{+}{-}{+}{+}
\end{scope}

\begin{scope}[shift={(0,-2.6)}]
    \signedsquarediag{-}{+}{+}{-}{+}
\end{scope}

\begin{scope}[shift={(2.4,-2.6)}]
    \signedsquarediag{-}{+}{+}{+}{-}
\end{scope}

\begin{scope}[shift={(4.8,-2.6)}]
    \signedsquarediag{+}{-}{-}{+}{+}
\end{scope}

\begin{scope}[shift={(7.2,-2.6)}]
    \signedsquarediag{+}{-}{+}{-}{+}
\end{scope}

\begin{scope}[shift={(9.6,-2.6)}]
    \signedsquarediag{+}{-}{+}{+}{-}
\end{scope}

\begin{scope}[shift={(12.0,-2.6)}]
    \signedsquarediag{+}{+}{-}{-}{+}
\end{scope}

\begin{scope}[shift={(14.4,-2.6)}]
    \signedsquarediag{+}{+}{-}{+}{-}
\end{scope}

\begin{scope}[shift={(16.8,-2.6)}]
    \signedsquarediag{+}{+}{+}{-}{-}
\end{scope}

\end{tikzpicture}
}

\caption{The sixteen equivalence classes of $\{\pm 1\}$-labelings on $\M$.}
\label{fig:g-labelings-example}
\end{figure}
\end{example}

Note that, for each group $G$ of size $q$ and each matroid $\M$ on $n$ elements, there are exactly $q^{n-c(\M)}$ equivalence classes of $G$-labelings of $\M$, where $c(\M)$ stands for the number of connected components of $\M$.

\subsection{Automorphisms}\label{ssec:automorphisms}

We first recall the structure of the group $\operatorname{Aut}(\Q_n(G))$ of all automorphisms of $\Q_n(G)$, following Bonin \cite{BoninAut}. When $G$ is trivial and $n \ge 2$, $\operatorname{Aut}(\Q_n(G))$ is $\mathfrak{S}_{n+1}$. If $G$ is nontrivial, then we will describe the automorphism group of $\Q_n(G)$ as a quotient of $(G \wr \mathfrak{S}_n) \rtimes \operatorname{Aut}(G)$, the semidirect product of the wreath product of $G$ with $\mathfrak{S}_n$ and the automorphism group of $G$. As a set, the group $(G \wr \mathfrak{S}_n) \rtimes \operatorname{Aut}(G)$ can be identified with $G^n \times \mathfrak{S}_n \times \operatorname{Aut}(G)$. The multiplication is given by
\begin{multline*}
((g_1, \dotsc, g_n), \sigma,  \phi) \cdot ((h_1, \dotsc, h_n), \tau, \psi) = \\((\phi(h_{1}) g_{\tau(1)}, \phi(h_2) g_{\tau(2)}, \dotsc, \phi(h_n) g_{\tau(n)}), \sigma \tau, \phi \psi),
\end{multline*} 
where $g_1, \dotsc, g_n, h_1, \dotsc, h_n$ lie in $G$, $\sigma$ and $\tau$ lie in $\mathfrak{S}_n$, and $\phi$ and $\psi$ lie in $\operatorname{Aut}(G)$. 

We describe the action of the group $(G \wr \mathfrak{S}_n) \rtimes \operatorname{Aut}(G)$ on partial $G$-partitions. Given a partial $G$-partition $(S, f, A_1, \dotsc, A_k)$ and an element $((g_1, \dotsc, g_n), \sigma,  \phi)$, we have an associated partial $G$-partition given by $\sigma(S) = \sigma(A_1) \sqcup \dotsb \sqcup \sigma(A_k)$ and a function given by
\begin{equation}\label{eq:action}
(((g_1, \dotsc, g_n), \sigma,  \phi) \cdot f)(i) = \phi(f(\sigma^{-1}(i)))g_{\sigma^{-1}(i)}.
\end{equation}

This action descends to an action on equivalence classes of partial $G$-partitions which respects the partial order. 
Let $N$ be the subgroup of $(G \wr \mathfrak{S}_n) \rtimes \operatorname{Aut}(G)$ given by elements $\{((g, \dotsc, g), \operatorname{id}, \operatorname{inn}_g)\}$, where $\operatorname{inn}_g$ is the automorphism given by $\operatorname{inn}_g(h) = g h g^{-1}$. Then $N$ is a normal subgroup of $(G \wr \mathfrak{S}_n) \rtimes \operatorname{Aut}(G)$,  and $N$ acts trivially on equivalence classes of partial $G$-partitions. 

Recall that $\Gamma$ is the image of $(G \wr \mathfrak{S}_n) \rtimes \operatorname{Aut}(G)$ in $\operatorname{Aut}(\Q_n(G))$. 
In \cite[Theorem~2.1]{BoninAut}, Bonin shows that if $|G| > 1$ and $n \ge 3$, then $\operatorname{Aut}(\Q_n(G)) = \Gamma$. 
The group $\operatorname{Aut}(\Q_n(G))$ was further studied in \cite{Schwartz}. 

\begin{example}
Let $G = \mathbb{Z}/2\mathbb{Z}$. Then $(G \wr \mathfrak{S}_n) \rtimes \operatorname{Aut}(G) = \left(\mathbb{Z}/2\mathbb{Z}\right)^n \rtimes \mathfrak{S}_n$ is the Weyl group of type $B_n$. The element $((-1, \dotsc, -1), \operatorname{id})$ acts nontrivially on the vector space underlying the type $B_n$ braid arrangement, but it preserves all of the hyperplanes, so the automorphism group of the type $B_n$ braid matroid is the quotient of the Weyl group of type $B_n$ by the subgroup generated by this element. 
\end{example}

There is an action of $\Gamma$ on equivalence classes of $G$-labeled matroids. First, note that there is an action of $(G \wr \mathfrak{S}_n) \rtimes \operatorname{Aut}(G)$ on the set of all functions $f \colon [n] \to G$ given by the same formula as in \eqref{eq:action}. 
The action by the subgroup $N$ preserves equivalence, so this descends to an action of $\Gamma$ on equivalence classes. 

\subsection{Equivariant Kazhdan--Lusztig polynomials}

We recapitulate the definition of equivariant KL and $Z$-polynomials. This definition was first stated in this form in \cite[Appendix~A]{braden-huh-matherne-proudfoot-wang2}, but the objects were considered in \cite{gedeon-proudfoot-young,Equivariant}. For a group $H$, let $\VRep(H)$ be the ring of virtual representations of $H$. 

\begin{definition}\label{def:equivariantKL}
    There is a unique assignment associating to each action $H \curvearrowright \M$ on a loopless matroid\footnote{The action is required to send flats to flats.} a polynomial $P_\M^H(t) \in \VRep(H)[t]$ satisfying
\begin{enumerate}[\normalfont(i)]
    \item If $\rk(\M) = 0$, then $P_\M^{H}(t) = \operatorname{tr}$ is the trivial representation.
    \item If $\rk(\M) > 0$, then $\deg P_\M^H(t) < \frac{1}{2}\rk(\M)$.
    \item For every $\M$, the polynomial
    \[Z_\M^H(t) = \sum_{[F] \in \mathcal{L}(\M)/H} t^{\rk(F)} \Ind_{H_F}^H P_{\M/F}^{H_F}(t)\]
    is palindromic of degree $\rk(\M)$.
\end{enumerate}
Here, $H_F$ denotes the stabilizer of $F$, and $\mathcal{L}(\M)/H$ is the set of orbits of flats of $\M$ under the action of $H$. The polynomial $P^{H}_{\M}(t)$ is called the \emph{equivariant Kazhdan--Lusztig polynomial}, and $Z^{H}_{\M}(t)$ is called the \emph{equivariant $Z$-polynomial}. 
\end{definition}

\section{Generating functions}

We recapitulate some results about the enumeration of quasi series-parallel matroids obtained in our prequel \cite{ferroni-larson}. Further computational results were obtained by Proudfoot, Xu, and Young in \cite{proudfoot-xu-young-qsp}. Throughout we shall assume that the reader has some acquaintance with generating functions and refer to \cite{stanley-ec2} for further details. This section is not used in the proof of Theorem~\ref{thm:main}, but it provides a powerful tool for computing Kazhdan--Lusztig and $Z$-polynomials of Dowling geometries.

\begin{proposition}[{\cite[Proposition~2.3]{ferroni-larson}}]
    Let $\mathcal{C}(n,k)$ denote the set of series-parallel matroids on $[n]$ of rank $k$. Then the bivariate generating function 
    \[ C(x,y) := \sum_{n=1}^{\infty} \sum_{k=0}^n |\mathcal{C}(n,k)| \frac{x^n}{n!} \, y^k ,\]
    is given by
    \begin{equation}\label{eq:C-gen-fun}
        C(x,y) = x(y+1) + y\int \left[ \left( \frac{1}{y} \log(1+xy) + \log(1+x) - x\right)^{\left<-1\right>}\right] dx.
    \end{equation}
    Here the symbol $\left<-1\right>$ stands for the compositional inverse of the function inside the parenthesis with respect to the variable $x$, i.e., treating $y$ as a constant.
\end{proposition}

This generating function allowed us to deduce in \cite[Proposition~2.8]{ferroni-larson} a bivariate generating function $A(x,y) = \exp(C(x,y))$ for the number of quasi series-parallel matroids on a fixed ground set and fixed rank, in particular also obtaining a generating function for the $Z$-polynomial of braid matroids depending on their rank.

The following result provides a generalization for arbitrary Dowling geometries.

\begin{proposition}\label{prop:gen-fun-AG}
    Let $\mathcal{A}_G(n,k)$ denote the set of all $G$-labeled quasi series-parallel matroids on $[n]$ of rank $k$. Then the bivariate generating function 
    \[ A_G(x,y) := \sum_{n=0}^{\infty} \sum_{k=0}^n |\mathcal{A}_G(n,k)| \frac{x^n}{n!} \, y^k ,\]
    is given by
    \begin{equation}\label{eq:A-gen-fun}
        A_G(x,y) = \exp(C(qx,y))^{1/q} = A(qx,y)^{1/q},
    \end{equation}
    where $C(x,y)$ is given as in equation~\eqref{eq:C-gen-fun} and $q = |G|$.
\end{proposition}

\begin{proof}
Let
\[
C_G(x,y):=\sum_{n\ge 1}\sum_{k=0}^n c_G(n,k)\frac{x^n}{n!}\,y^k
\]
be the bivariate exponential generating function for connected $G$-labeled quasi
series-parallel matroids. By \cite[Proposition~2.8]{ferroni-larson},
\[
A(x,y)=\exp(C(x,y)),
\]
so $C(x,y)$ is the connected part of the species of quasi series-parallel matroids.

If $\M$ is connected and has $n$ elements, then $\M$ has exactly $q^{n-1}$ distinct
$G$-labelings, since two labelings are identified precisely when they differ by
left multiplication by a single element of $G$ on the unique connected component.
Therefore
\[
C_G(x,y)=\frac{1}{q}\,C(qx,y).
\]
Now an arbitrary $G$-labeled quasi series-parallel matroid can be constructed from a set of connected
ones, so the exponential formula gives
\[
A_G(x,y)=\exp(C_G(x,y))
=\exp\!\left(\frac{1}{q}C(qx,y)\right)
=\exp(C(qx,y))^{1/q}
=A(qx,y)^{1/q}.\qedhere\]
\end{proof}

If we restrict our enumeration to $G$-labeled \emph{simple} quasi series-parallel matroids, we were led\footnote{Actually, in \cite[Proposition~2.14]{ferroni-larson} we were led to the generating function $S(x,y) = \frac{1}{x+1} A(\log(x+1),y) - 1$, that is, we did not count the empty matroid as a simple quasi series-parallel matroid. This difference is inessential for our computations, and our subsequent formulas will be much simpler if we include the empty matroid in the count.} in \cite[Proposition~2.14]{ferroni-larson} to the generating function $S(x,y) = \frac{1}{x+1}A(\log(x+1),y)$, and therefore to the generating function for Kazhdan--Lusztig polynomials of braid matroids. We have the following generalization to arbitrary Dowling geometries.

\begin{proposition}\label{prop:gen-fun-S}
    Let $\mathcal{S}_G(n,k)$ denote the set of all $G$-labeled simple quasi series-parallel matroids on $[n]$ of rank $k$. Then the bivariate generating function 
    \[ S_G(x,y) := \sum_{n=0}^{\infty} \sum_{k=0}^n |\mathcal{S}_G(n,k)| \frac{x^n}{n!} \, y^k ,\]
    is given by
    \begin{equation}\label{eq:S-gen-fun}
        S_G(x,y) = \left(\frac{1}{qx+1} A(\log(qx+1),y)\right)^{1/q} = S(qx,y)^{1/q},
    \end{equation}
    where $A(x,y)=\exp(C(x,y))$ and $q = |G|$.
\end{proposition}

\begin{proof}
Let
\(
T(x,y):=\log S(x,y).
\)
Since a simple quasi series-parallel matroid is a direct sum of connected simple
quasi series-parallel matroids, $T(x,y)$ is the exponential generating function
for connected simple quasi series-parallel matroids.

As in the proof of Proposition~\ref{prop:gen-fun-AG}, a connected matroid on $n$
elements has exactly $q^{n-1}$ distinct $G$-labelings. Hence the exponential
generating function for connected $G$-labeled simple quasi series-parallel
matroids is
\(
T_G(x,y)=\frac{1}{q}\,T(qx,y).
\)
Applying the exponential formula again, we obtain
\[
S_G(x,y)=\exp(T_G(x,y))
=\exp\!\left(\frac{1}{q}\log S(qx,y)\right)
=S(qx,y)^{1/q}.
\]
Finally, \cite[Proposition~2.14]{ferroni-larson} gives
\[
S(x,y)=\frac{1}{1+x}A(\log(1+x),y),
\]
and therefore
\[
S_G(x,y)
=S(qx,y)^{1/q}
=\left(\frac{1}{1+qx}A(\log(1+qx),y)\right)^{1/q},
\]
as claimed.
\end{proof}

\section{Proof of Theorem~\ref{thm:main}}

We now recall some properties of the flats of $\Q_n(G)$, following Dowling \cite[Theorem 2]{dowling}. Recall that a flat $F$ of $\Q_n(G)$ corresponds to a partial $G$-partition $(S, f, A_1, \dotsc, A_k)$. The rank of this flat is $n - k$. Let $n_i = |A_i|$, and set $n_0 = n - n_1 - \dotsb - n_k$.
The restriction of $\Q_n(G)$ to a flat is isomorphic to $\Q_{n_0}(G) \oplus \K_{n_1} \oplus \dotsb \oplus \K_{n_k}$. Indeed, let $A_0 = [n] \setminus (A_1 \sqcup \dotsb \sqcup A_k)$. Then an equivalence class of partial $G$-partitions which is less than or equal to $(S, f, A_1, \dotsc, A_k)$ corresponds to an equivalence class of partial $G$-partitions of $A_0$ and a partition of $A_i$ for $i = 1, \dotsc, k$. 

We call the integer $n_0$ and the multiset $\{n_1, \dotsc, n_k\}$ the  \emph{type} of the flat. Note that $\Gamma$ acts transitively on the flats of a particular type. 

The simplification of the contraction of a flat $F$ which is isomorphic to $\Q_{n_0}(G) \oplus \K_{n_1} \oplus \dotsb \oplus \K_{n_k}$ is isomorphic to $\Q_k(G)$. Indeed, giving an equivalence class of partial $G$-partitions which is greater than or equal to $(S, f, A_1, \dotsc, A_k)$ amounts to giving an element of $G$ for each of $A_1, \dotsc, A_k$ and merging some of the parts. After modding out by equivalence, this is the same thing as an equivalence class of partial $G$-partitions on $[k]$. 
There is an induced map from $\Gamma_F$, the stabilizer of $F$ in $\Gamma$, to $\operatorname{Aut}(\Q_k(G))$ whose image is the image of $(G \wr \mathfrak{S}_k) \rtimes \operatorname{Aut}(G)$ in $\operatorname{Aut}(\Q_k(G))$. That is, if $k\ge 3$, then $\Gamma_F$ surjects onto $\operatorname{Aut}(\Q_k(G))$ if $G$ is nontrivial, and its image is isomorphic to $\mathfrak{S}_k$ if $G$ is trivial. 

\begin{proof}[Proof of Theorem~\ref{thm:main}]
We induct on $n$. Let $P(t) \in \operatorname{VRep}(\Gamma)[t]$ be the generating function of the permutation representations of $\Gamma$ on $G$-labeled simple quasi series-parallel matroids of rank $n - i$ on $[n]$, and let $Z(t) \in \operatorname{VRep}(\Gamma)[t]$ be the generating function of the permutation representations of $\Gamma$ on $G$-labeled quasi series-parallel matroids of rank $n - i$ on $[n]$. By \cite[Proposition 2.10]{ferroni-larson}, a simple quasi series-parallel matroid on $[n]$ has rank greater than $n/2$, so the degree of $P(t)$ is less than $n/2$. As $Z(t)$ is palindromic because the dual of a quasi series-parallel matroid is quasi series-parallel, it suffices to verify that
\begin{equation}\label{eq:identity}
Z(t) = P(t) + \sum_{\substack{[F] \in \mathcal{L}(\Q_n(G))/\Gamma\\ [F]\neq \hat{0}}} t^{\operatorname{rk}(F)} \operatorname{Ind}^{\Gamma}_{\Gamma_F} P_{\Q_n(G)/F}(t).
\end{equation}
As mentioned above, the orbits of the action of $\Gamma$ on the lattice of flats of $\Q_n(G)$ are given by types. Choose a flat $F$ which is isomorphic to $\Q_{n_0}(G) \oplus \K_{n_1} \oplus \dotsb \oplus \K_{n_k}$, and let $[n] = S_0 \sqcup \dotsb \sqcup S_k$ be the corresponding partition. By induction, $P_{\Q_n(G)/F}(t)$ is the permutation representation of $\Gamma_F$ on $G$-labeled simple quasi series-parallel matroids on $[k]$; by \cite[Proposition 3.1]{ferroni-larson}, the equivariant Kazhdan--Lusztig polynomial is unchanged by simplification. 

Given a subgroup $K$ of $\Gamma_F$, the induction of the permutation representation of $\Gamma_F$ on $\Gamma_F/K$ to $\Gamma$ is the permutation representation of $\Gamma$ on $\Gamma/K$. Given an orbit of equivalence classes of $G$-labeled simple quasi series-parallel matroids of rank $k - i$ on $[k]$, we will produce an orbit of equivalence classes of $G$-labeled quasi series-parallel matroids of rank $k -  i = n - i - \operatorname{rk}(F)$ on $[n]$ with the same stabilizer. 

Because $\Gamma_F$ and $\Gamma$ act transitively on the set of equivalence classes of $G$-labelings of a fixed quasi series-parallel matroid on $[k]$ or $[n]$ elements, respectively, it suffices to construct a quasi series-parallel matroid on $[n]$ from a simple quasi series-parallel matroid on $[k]$. Given a quasi series-parallel matroid $\M$ of rank $k - i$ on $[k]$, construct a matroid $\tilde{\M}$ on $[n]$ with set of loops given by $S_0$ and parallel classes given by $S_1, \dotsc, S_k$ which simplifies to $\M$. Note that $\tilde{\M}$ is quasi series-parallel and has rank $k - i$. 

For each term on the right-hand side of \eqref{eq:identity}, we have produced a matching term on the left-hand side. Furthermore, every quasi series-parallel matroid on $[n]$ induces a partition $[n] = S_0 \sqcup \dotsb \sqcup S_k$ and a simple quasi series-parallel matroid on $[k]$, so this proves \eqref{eq:identity}. 
\end{proof}

\begin{example}
Let $n=3$. Then we can describe $\Q_n(G)$ as follows: it is a rank $3$ simple matroid with ground set $\{b_1, b_2, b_3\} \cup \{g_{i,j}\}_{1 \le i < j \le 3, \, g \in G}$. The rank $2$ flats are as follows:
\begin{enumerate}
\item There are three flats isomorphic to $\Q_2(G)$, given by $\{b_i, b_j\}  \cup \{g_{i,j}\}_{g \in G}$. 
\item There are $3q$ flats isomorphic to $\Q_1(G) \oplus \mathsf{K}_2$, given by $\{b_i, g_{j,k}\}$, where $i \not= j, k$ and $g \in G$. 
\item There are $q^2$ flats isomorphic to $\K_3$, given by $\{g_{12}, h_{23}, (g h)_{13}\}$ for $g, h \in G$. 
\end{enumerate}
See Figure~\ref{fig:n=3}. 
\begin{figure}[hb]
\begin{tikzpicture}[scale=0.85,auto=center,
    bvertex/.style={circle, fill=black, inner sep=2.2pt},
    epoint/.style={circle, fill=black, inner sep=1.6pt},
    upoint/.style={circle, fill=black, inner sep=1.6pt}
]
 
\coordinate (b1) at (0,0);
\coordinate (b2) at (7,0);
\coordinate (b3) at (3,6);
 
\coordinate (g12)  at ($(b1)!0.3!(b2)$);
\coordinate (x12)  at ($(b1)!0.7!(b2)$);
\coordinate (h23)  at ($(b2)!0.3!(b3)$);
\coordinate (y23)  at ($(b2)!0.7!(b3)$);
\coordinate (gh13) at ($(b3)!0.3!(b1)$);
\coordinate (xy13) at ($(b3)!0.7!(b1)$);
 
\coordinate (u12a) at ($(b1)!0.1!(b2)$);
\coordinate (u12b) at ($(b1)!0.5!(b2)$);
\coordinate (u12c) at ($(b1)!0.9!(b2)$);
\coordinate (u23a) at ($(b2)!0.1!(b3)$);
\coordinate (u23b) at ($(b2)!0.5!(b3)$);
\coordinate (u23c) at ($(b2)!0.9!(b3)$);
\coordinate (u13a) at ($(b3)!0.1!(b1)$);
\coordinate (u13b) at ($(b3)!0.5!(b1)$);
\coordinate (u13c) at ($(b3)!0.9!(b1)$);
 
\draw[thick] (b1) -- (b2) -- (b3) -- cycle;
 
\draw[thick]
    plot[smooth cycle, tension=0.78]
    coordinates {(g12) (h23) (gh13)};
 
\draw[thick]
    plot[smooth cycle, tension=0.78]
    coordinates {(x12) (y23) (xy13)};
 
\node[bvertex, label={below left:$b_1$}]  at (b1) {};
\node[bvertex, label={below right:$b_2$}] at (b2) {};
\node[bvertex, label={above:$b_3$}]       at (b3) {};
 
\node[epoint, label={below:$g_{12}$}]        at (g12)  {};
\node[epoint, label={below:$x_{12}$}]        at (x12)  {};
\node[epoint, label={right:$h_{23}$}]        at (h23)  {};
\node[epoint, label={right:$y_{23}$}]        at (y23)  {};
\node[epoint, label={left:$(gh)_{13}$}]      at (gh13) {};
\node[epoint, label={left:$(xy)_{13}$}]      at (xy13) {};
 
\node[upoint] at (u12a) {};
\node[upoint] at (u12c) {};
 
\node[upoint] at (u23a) {};
\node[upoint] at (u23c) {};
 
\node[upoint] at (u13a) {};
\node[upoint] at (u13c) {};
 
\end{tikzpicture}
\caption{The Dowling geometry $\Q_3(G)$.}\label{fig:n=3}
\end{figure}
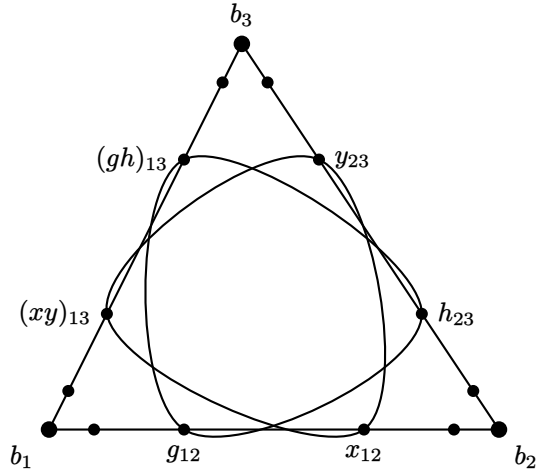
The coefficient of $t^{\operatorname{rk}(\M) - 1}$ in $Z_{\M}(t)$ is the number of hyperplanes. 
The first group of flats correspond to the equivalence classes of $G$-labelings of the rank $1$ quasi series-parallel matroids isomorphic to $\U_{0,2} \oplus \U_{1,1}$. The second group corresponds to $\U_{0,1} \oplus \U_{1,2}$. The third group corresponds to $\U_{1,3}$. The coefficient of $t$ in $P_{\mathrm{M}}(t)$ is the number of hyperplanes minus the number of atoms, which is equal to the number of flats isomorphic to $\K_3$. 
In \cite[Section 5]{dowling}, Dowling extends the definition of $\Q_3(G)$ to a more general algebraic object called a \emph{quasigroup}. The above analysis gives a description of the $Z$- and Kazhdan--Lusztig polynomial for Dowling geometries of quasigroups. 
\end{example}

\begin{example}
We can use Theorem~\ref{thm:main} to give a simple formula for the leading coefficient of $P_{\Q_n(G)}(t)$ when $n = 2m-1$ is odd. A simple quasi series-parallel matroid is a direct sum of connected simple quasi series-parallel matroids. As a simple quasi series-parallel matroid on $[n]$ has rank greater than $n/2$, this implies that a simple quasi series-parallel matroid of rank $m$ on $[2m-1]$ is connected. It was shown in \cite[Corollary 2.12]{ferroni-larson} that there are $(2m-3)!! (2m-1)^{m-2}$ simple quasi series-parallel matroids of rank $m$ on $[2m-1]$, so the coefficient of $t^{m-1}$ in $P_{\Q_{2m-1}(G)}(t)$ is $q^{2m-2} (2m-3)!! (2m-1)^{m-2}$. 

A simple quasi series-parallel matroid of rank $m+1$ on $[2m]$ has either one or two connected components. Connected simple quasi series-parallel matroids of rank $m+1$ on $[2m]$ were enumerated in \cite[Corollary 1.6]{gao-proudfoot-yang-zhang}, so Theorem~\ref{thm:main} can be used to give an explicit formula for the leading coefficient of $P_{\Q_{2m}(G)}(t)$, but the resulting formula is not very clean. See \cite[Proposition 2.13]{ferroni-larson}. 
\end{example}

\section{Final Remarks}

\subsection{Inverse Kazhdan--Lusztig polynomials}

In \cite{gao-xie} Gao and Xie introduced \emph{inverse Kazhdan--Lusztig polynomials} of matroids, and Proudfoot introduced an equivariant generalization in \cite{Equivariant}. The inverse Kazhdan--Lusztig polynomial of $\M$ can be computed (up to a sign) by inverting the Kazhdan--Lusztig function in the incidence algebra of $\mathcal{L}(\M)$ and specializing to the interval $[\widehat{0},\widehat{1}]$. 

Braden, Huh, Matherne, Proudfoot, and Wang proved in \cite[Theorem~1.5]{braden-huh-matherne-proudfoot-wang2} that inverse Kazhdan--Lusztig polynomials of matroids have nonnegative coefficients. Furthermore, they proved that the coefficients of the equivariant inverse Kazhdan--Lusztig polynomial associated to $H \curvearrowright \M$ are isomorphism classes of honest, rather than virtual, representations of $H$.

The following remains an unsolved problem that we regard as especially interesting.

\begin{problem}
    Find a combinatorial interpretation for the coefficients of the inverse Kazhdan--Lusztig polynomial of Dowling geometries.
\end{problem}

The above question is open even in the case of braid matroids. However, in this special case some progress has been made: Gao, Proudfoot, Yang, and Zhang \cite{gao-proudfoot-yang-zhang} interpreted the \emph{leading} coefficient of the inverse Kazhdan--Lusztig polynomial of $\K_n$.

An analogous question can be formulated about the inverse $Z$-polynomial introduced in \cite{FMSV}, recently studied by Gao, Ruan, and Xie \cite{gao-ruan-xie} and equivariantly by Gao, Li, and Xie \cite{gao-li-xie}.

\subsection{Real-rootedness}

Two striking conjectures by Gedeon, Proudfoot, and Young \cite{gedeon-proudfoot-young-survey} and by Proudfoot, Xu, and Young \cite{PXY} assert that, for every matroid $\M$, the polynomials $P_{\M}(t)$ and $Z_{\M}(t)$ have only nonpositive real zeros. The results in the present paper can be used to test this conjecture.

From our generating functions, it is straightforward to deduce that the polynomials $P_{\Q_n(G)}(t)$ have integer coefficients after the change of variables $t = t/q^2$. Furthermore, one can compute them explicitly for relatively small values of $n$:

\begin{equation}   P_{\Q_n(G)}\left(t/q^2\right) = \begin{cases}
        1 & n = 1, \\
1 & n = 2, \\
t + 1 & n = 3, \\
{\left(q + 4\right)} t + 1 & n = 4, \\
15 t^{2} +{\left(q^{2} + 5 q + 10\right)} t +  1 & n = 5, \\
25 {\left(3 q + 4\right)} t^{2} + {\left(q^{3} + 6 q^{2} + 15 q + 20\right)} t + 1 & n = 6, \\ 735 t^{3} +
35 {\left(8 q^{2} + 16 q + 11\right)} t^{2} &\\+ {\left(q^{4} + 7 q^{3} + 21 q^{2} + 35 q + 35\right)} t + 1 & n = 7, \\
105 {\left(89 q + 64\right)} t^{3} + 7 {\left(134 q^{3} + 333 q^{2} + 340 q + 160\right)} t^{2} \\ 
\qquad + {\left(q^{5} + 8 q^{4} + 28 q^{3} + 56 q^{2} + 70 q + 56\right)} t + 1 & n = 8.
\end{cases} \label{eq:kl-small-n}
\end{equation}

Similar computations can be made for the $Z$-polynomial of $\Q_n(G)$ for relatively small $n$, except that $Z_{\Q_n(G)}(t/q^2)$ does not have integer coefficients.

\begin{proposition}
    For all $n\leq 15$ and every group $G$, the polynomials $P_{\Q_n(G)}(t)$ and $Z_{\Q_n(G)}(t)$ are real-rooted. 
\end{proposition}

The proof of this result is a direct computation that in fact yields a stronger fact: for each $n\leq 14$, the polynomial $P_{\Q_n(G)}(t/q^2)$ interlaces $P_{\Q_{n+1}(G)}(t/q^2)$ and, similarly, $Z_{\Q_{n}(G)}(t)$ interlaces $Z_{\Q_{n+1}(G)}(t)$.

To see this, a useful (finite) criterion for interlacing is given by the B\'ezoutian\footnote{We learned of this idea thanks to a talk given by Nick Proudfoot at Banff in March 2023.}. Consider two polynomials $f(x),g(x)\in \mathbb{R}_{\geq 0}[x]$ of degrees $d$ and $e\in \{d,d-1\}$. The \emph{B\'ezoutian}, or the \emph{B\'ezout matrix}, associated to the pair $(f,g)$ is the symmetric matrix
    \[ B(f,g) := \left( a_{ij}\right)_{0\leq i,j \leq d-1}, \quad \text{ where $a_{ij}= [x^iy^j]\left(\frac{f(x)g(y)-f(y)g(x)}{x-y}\right)$}. \]
A classical result in the theory of interlacing polynomials (see \cite[Section~2]{krein-naimark}) asserts that $g$ (strictly) interlaces $f$ if and only if $B(f,g)$ is positive definite. 

For a fixed group $G$ of size $q$, let us denote $P_n(x) = P_{\Q_n(G)}(x/q^2)$ and $Z_n(x) = Z_{\Q_n(G)}(x)$. We have verified with the help of a computer that the following holds.

\begin{proposition}\label{prop:bezoutian-totally-positive}
    For each $n\leq 14$, the matrices $B(P_{n+1},P_{n})$ and $B(Z_{n+1},Z_n)$ are totally positive. Moreover, all the minors of these matrices are polynomials in $q-1$ with positive coefficients. 
\end{proposition}   

Since total positivity is a (much) stronger condition than positive-definiteness, this shows that our polynomials are real-rooted and interlace in a very strong sense.

\begin{example}
    Let us illustrate Proposition~\ref{prop:bezoutian-totally-positive}. For $n=5$ we have: 
    \[B(P_6,P_5) = \begin{pmatrix}
    q^{3}+5q^{2}+10q+10 & 75q+85 \\
    75q+85 & 60q^{3}+385q^{2}+1025q+700
\end{pmatrix} 
    \]
    For each $q\geq 1$, this matrix is totally positive because it has strictly positive entries and its determinant is equal to
    \begin{align*}
        \det B(P_6,P_5) &= 60 \, q^{6} + 685 \, q^{5} + 3550 \, q^{4} + 10275 \, q^{3} + 11975 \, q^{2} + 4500 \, q - 225\\
            &= 60 \, (q-1)^{6} + 1045 \, (q-1)^{5} + 7875 \, (q-1)^{4} + 32525 \, (q-1)^{3} + \\
            &\qquad 71850 \, (q-1)^{2} + 77260 \, (q-1) + 30820.
    \end{align*}
\end{example}

The threshold of $n\leq 15$ in our results can likely be pushed further with additional computational power. It is unclear, however, how to exploit this idea to prove the real-rootedness of $P_{\Q_n(G)}(x)$ and $Z_{\Q_n(G)}(x)$ for arbitrary $n$ (and arbitrary $q$).

\begin{remark}
    It is tempting to ask about the meaning of the coefficients and the general behavior of the polynomial appearing in equation~\eqref{eq:kl-small-n} after substituting $q = 0$. Note that $P_7(x) = 1 +35 x+ 385x^2 + 735x^3$ is not real-rooted.
\end{remark}

\subsection{Intersection cohomology of Dowling lattices}

One of the main results of \cite{braden-huh-matherne-proudfoot-wang2} states that $Z_{\M}(t)$ is the Poincar\'{e} polynomial of a certain graded vector space denoted $\IH(\M)$, which is called the \emph{intersection cohomology} of $\M$. This graded vector space is a module over the graded M\"obius algebra $\H(\M)$. Theorem~\ref{thm:main} suggests that $\IH(\Q_n(G))$ has a basis indexed by equivalence classes of $G$-labelings of quasi series-parallel matroids. 
Note that each flat of rank $k$ of $\Q_n(G)$ naturally gives rise to a $G$-labeled quasi series-parallel matroid on $[n]$ of rank $n-k$: if the partition is $S = A_1 \sqcup \dotsb \sqcup A_{n-k}$, then we take the direct sum of uniform matroids of rank $1$ on $A_1, \dotsc, A_{n-k}$ with loops on $[n] \setminus S$.  
Dowling showed that, if $n \ge 3$, then $\Q_n(G)$ is realizable over a field if and only if $G$ is cyclic \cite[Theorem 11]{dowling}, so this would give an explicit understanding of the intersection cohomology of some non-realizable matroids. 

After the publication of \cite{ferroni-larson}, several people attempted to construct a basis for $\IH(\K_{n+1})$ indexed by quasi series-parallel matroids. We recount one suggestion by June Huh which has some interesting consequences. Because $\K_{n+1}$ is a binary matroid, it is a deletion of the binary full projective geometry $\mathbb{P}_{\mathbb{F}_2}^{n-1}$. This implies that $\IH(\mathbb{P}^{n-1}_{\mathbb{F}_2})$ is a module over $\H(\K_{n+1})$, and that $\IH(\K_{n+1})$ is an $\H(\K_{n+1})$-module summand. 

Because the matroid $\mathbb{P}^{n-1}_{\mathbb{F}_2}$ is modular, $\IH(\mathbb{P}^{n-1}_{\mathbb{F}_2}) = \H(\mathbb{P}^{n-1}_{\mathbb{F}_2})$, and so it has a basis labeled by subspaces of $\mathbb{F}_2^{n}$. Because binary matroids are uniquely representable, each quasi series-parallel matroid on $[n]$ gives rise to an element of $\IH(\mathbb{P}^{n-1}_{\mathbb{F}_2})$. However, the subspace of $\IH(\mathbb{P}^{n-1}_{\mathbb{F}_2})$ generated by classes of quasi series-parallel matroids is not preserved by the action of $\H(\K_{n+1})$ for $n \ge 7$. 

Because all binary matroids of rank at most $2$ are quasi series-parallel, $\IH^i(\K_{n+1}) = \IH^i(\mathbb{P}^{n-1}_{\mathbb{F}_2})$ for $i \le 2$. 
There is an action of $\mathfrak{S}_{n+1}$ on $\IH(\mathbb{P}^{n-1}_{\mathbb{F}_2})$ because we can view $\mathbb{F}_2^n$ as the standard representation of $\mathfrak{S}_{n+1}$. This implies that $\IH^i(\K_{n+1})$ is a permutation representation of $\mathfrak{S}_{n+1}$ for $i \le 2$.

\bibliographystyle{amsalpha}
\bibliography{bibliography}
\end{document}